\documentclass{amsart}
\usepackage{graphicx}

\usepackage{amsmath,amsfonts,amsthm,amssymb,graphics,color}

\usepackage{enumerate}

        \definecolor{pink}{rgb}{1,0,1}

\usepackage[abbrev]{amsrefs}

\newtheorem{theorem}{Theorem}
\newtheorem{prop}[theorem]{Proposition}
\newtheorem{ex}{Exercise}
\newtheorem{lemma}[theorem]{Lemma}

\newtheorem{question}[theorem]{Question}

\newtheorem{conj}{Conjecture}
\theoremstyle{remark}
\newtheorem{remark}{Remark}
\theoremstyle{definition}
\newtheorem{defn}[theorem]{Definition}

\numberwithin{equation}{section}

\newcommand{\M}{\mathfrak{M}}

\newcommand{\pa}{\partial}
\newcommand{\eps}{\varepsilon}

\newcommand{\N}{\mathbb{N}}
\newcommand{\R}{\mathbb{R}}

\newcommand{\cC}{\mathcal{C}}

\newcommand{\D}{\mathbb{D}}

\newcommand{\vphi}{\varphi}

\title{The sound of symmetry}
\author[Zhiqin Lu]{Zhiqin Lu}\address{Department of Mathematics, 410D Rowland Hall - University of California, Irvine, CA 92697-3875.} \email{zlu@uci.edu}

\author[Julie Rowlett]{Julie Rowlett} \address{Max Planck Institut f\"ur Mathematik, Vivatgasse 7, 53111 Bonn, Germany.} \curraddr{Mathematical Sciences, Chalmers University and the University of Gothenburg, 41296, Sweden.\\ \url{http://www.math.chalmers.se/~rowlett/}} \email{julie.rowlett@chalmers.se} 


\begin{document} 
\begin{abstract}  This note begins with an introduction to the inverse isospectral problem popularized by M. Kac's 1966 article in the American Mathematical Monthly, ``Can one hear the shape of a drum?" Although the answer has been known for some twenty years now, many open problems remain.  Intended for general audiences, readers are challenged to complete exercises throughout this interactive introduction to inverse spectral theory.   Following the introduction, the main techniques used in inverse isospectral problems are collected and discussed.  These are then used to prove that one \em can \em hear the shape of:  parallelograms, acute trapezoids, and the regular n-gon.  Finally, we show that one can \em realistically \em hear the shape of the regular n-gon amongst all convex n-gons because it is uniquely determined by a finite number of eigenvalues;  the sound of symmetry can really be heard!  
\end{abstract}

\maketitle

\tableofcontents

\section{Introduction}

Have you heard the question, ``Can one hear the shape of a drum?''  Do you know the answer?  This question is the title of an article published in 1966 by M. Kac \cite{kac}  based on the following.  
\begin{question} 
If two planar domains have the same spectrum, are they identical up to rigid motions of the plane?  
\end{question} 

For a domain $\Omega$ in $\R^2$, the above mentioned \em spectrum \em is the set of eigenvalues of the Laplace operator $\Delta$ with Dirichlet boundary condition.  This is the set of all real numbers $\lambda$ such that there exists a smooth function $u$ on $\bar \Omega$ satisfying the Laplace equation and boundary condition 
\begin{equation} \label{lap} \Delta u(x,y) := \frac{\pa^2 u}{\pa x^2} + \frac{\pa^2 u }{\pa y^2} = - \lambda u(x,y), \quad u = 0 \textrm{ on the boundary of } \Omega. \end{equation} 
The Laplace equation originates from the wave equation 
$$\Delta f(x, y, t)=c^2\frac{\pa^2 f (x,y,t)}{\pa t ^2},$$
where $c$ is a positive constant depending on the elasticity of the vibrating material.  Identifying a domain $\Omega \subset \R^2$ with a vibrating drumhead, the function $f(x,y,t)$ gives the height of the drumhead at a point $(x,y) \in \Omega$ at time $t$, where $0$ is the height when the drum is not moving.  Since the boundary of the drumhead is fixed to the drum, which presumably is a solid material, this corresponds mathematically to the Dirichlet boundary condition.  Separating variables, that is assuming the solution $f$ can be expressed as 
$$f(x,y,t) = u(x,y)g(t),$$
it is straightforward to deduce ~\eqref{lap}.  

The eigenvalues of a bounded domain from a discrete set of $(0, \infty)$ 
$$0< \lambda_1 <\lambda_2 \leq \lambda_3 \ldots.$$
The set of eigenvalues $\{\lambda_k\}_{k=1} ^\infty$ is in bijection with the resonant frequencies a drum would produce if $\Omega$ were its drumhead.  With a perfect ear  one could hear all these frequencies and therefore know the spectrum.  Based on this physical description, Kac paraphrased Question 1 as, ``Can one hear the shape of a drum?''   In other words, if two drums sound identical to a perfect ear and therefore have identical spectra, then are they the same shape?

\subsection{Hearing a string}
Let's take a look at the simplest case and assume that our drum is actually a string, for example on a guitar or a violin.  If we hold the string so that its length is $\ell$ and then pluck it so that it vibrates, keeping the ends fixed, this is mathematically described by the ordinary differential equation
\[
f''(x) = - \lambda f(x),\quad f(0)=f(\ell)=0.
\]
We know from calculus that the solutions are 
\[
f_k (x) = \sin\left( \frac{k \pi x}{\ell} \right), \quad \textrm{ with } \lambda=\lambda_k=\frac{k^2\pi^2}{\ell^2}
\]
for $k=1,2,\cdots$.
\begin{question} Can one hear the shape of a string? \end{question} 
The ``shape'' of the string is just its length, so we can formulate the question as:  if we know the set of all $\{\lambda_k\}_{k=1} ^\infty$, then do we know the length of the string?  Well, it turns out that we actually only need to know $\lambda_1$ because then 
$$\ell = \sqrt{ \frac{\pi^2}{\lambda_1} }.$$
\begin{figure} 
  \caption{A Vibrating String} \includegraphics[width=200pt]{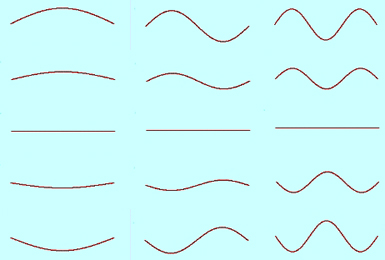}
\end{figure}

This shows that we can hear the length of the string based on the first eigenvalue.  Musically inclined readers are well aware of this fact because $\lambda_1$ determines the fundamental tone of the string.  The fact that $\lambda_1$ is in bijection with the length of the string mathematically corresponds to the fact that stringed instrument players can change the notes they play by holding the strings at different lengths.  

\subsection{The answer, open problems, and our results} 
Now that we have answered the question in one dimension, let's return to Kac's question in two dimensions.  Although perhaps the most natural drum is a circular drum, rectangular drums are a bit easier.  The Laplace equation for a rectangular domain is 
$$\frac{\pa^2 f}{\pa x^2} + \frac{\pa^2 f}{\pa y^2} = - \lambda f(x,y), \quad f(x,y) = 0 \textrm{ for } (x,y) \textrm{ on the boundary.}$$
Translating, rotating, or reflecting the domain doesn't change the numbers $\lambda$, so we can assume that the domain has vertices $(0,0)$, $(\ell, 0)$, $(0, w)$, $(\ell, w)$.  

\begin{ex}\label{e3} Prove that one \em can \em hear the shape of a rectangle.  
\end{ex}

Unfortunately it is only possible to compute the eigenvalues in closed form for a few special examples, such as rectangles and disks.  Without expressions for the eigenvalues, how can we answer Kac's question? 
Mathematically, the  question is equivalent to determining whether or not the following map
\[
\Lambda: \mathcal M\to \R^\infty, \quad \Omega\mapsto (\lambda_1,\lambda_2,\cdots,\lambda_k,\cdots),
\]
is injective, where $\mathcal M$ is the moduli space of all bounded domains (with piecewise smooth boundary) in $\R^2$. 

Kac's article appeared in print two years after a lovely one-page paper by Milnor \cite{milnor} which showed that \em one cannot hear the shape of a 16 dimensional drum.  \em  A flat torus is a Riemannian quotient manifold of the form $\R^n / L$ where $L$ is a lattice, that is a discrete additive subgroup of rank $n$.  Milnor used a construction of Witt \cite{witt} of self-dual lattices in $\R^{16}$ which are distinct in the sense that no rigid motion of $\R^{16}$ exists which maps one lattice to the other.  Consequently, Milnor was able to give a short proof that the corresponding tori are not isometric but have the same set of eigenvalues.   

Kac's original question however was for \em two dimensional \em domains, and consequently Milnor's result addresses the question for general dimensions but not for the specific case of two dimensions.  In 1985, Sunada introduced what he described as ``a geometric analogue of a routine method in number theory,'' which became known as \em the Sunada method \em \cite{sun} and can be used to produce large numbers of isospectral manifolds in four dimensions which are not isometric.  Buser generalized this method to construct isospectral non-isometric surfaces \cite{buser}.  However, this was still not quite the answer to Kac's question, since curved surfaces are not planar domains like the head of a drum.  

 Gordon, Webb, and Wolpert saw nonetheless that the basic ideas of Buser could be used to prove that the answer to Kac's question is ``no,'' by showing that the map $\Lambda$ is not injective on $\mathcal M$ \cites{gww,gww1}.  They proved that the two domains in Figure~\ref{drum} would sound identical to a perfect ear because they have identical spectrum, but as can be seen, the domains are not identical by rigid motion.     More recently Chapman published a charming article which shows how to construct isospectral non-isometric planar domains by folding paper \cite{chapman}.

\begin{figure} \caption{Identical sounding drums \cite{gww}}\label{drum}  \includegraphics[width=300pt]{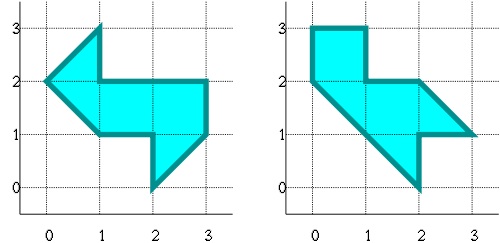} \end{figure}

Although it may seem that Kac's question was laid to rest in the 1990s, many open problems remain.
If we can't determine the shape of the domain completely by its spectrum, can we at least 
determine, or {\it hear}, some of  its geometric features like the convexity or smoothness of the boundary?  

 Motivated by these problems we shall investigate the injectivity of $\Lambda$ restricted to certain subsets of $\mathcal M$.  A natural choice is the set of convex $n$-gons since this set can be identified with a finite dimensional manifold with corners.  Surprisingly, even within this subset Kac's question is a subtle problem.  Durso proved \cite{dur} that one \em can \em hear the shape of a triangle, so in fact $\Lambda$ restricted to the moduli space of Euclidean triangles is injective (see also \cite{gm}).   For $n>3$ the problem is widely open.
 
  In this note, we present three theorems with simple albeit rather technical proofs.  Our  results seem to be new but more importantly, these proofs show most of the basic methods in inverse spectral theory in an elementary way.

 Our first result extends Durso's theorem to  parallelograms and acute trapezoids (see Definition ~\ref{dtrap}).  

\begin{theorem}\label{t1} If two parallelograms have the same spectrum, then they are identical up to rigid motions of the plane.  If two acute trapezoids have the same spectrum, then they are identical up to rigid motions of the plane.  
\end{theorem}

This shows that one \em can \em hear the shape of parallelograms and acute trapezoids.  

\begin{theorem} \label{t2} If an n-gon is isospectral to a regular n-gon, then  they are identical up to rigid motions of the plane.  
\end{theorem}

Theorem~\ref{t2} shows that the symmetry of the regular $n$-gon can be heard among all $n$-gons, which in the spirit of Kac we paraphrase as follows.  
\begin{quote}
 \em Among all $n$-gons, one can hear the symmetry of the regular one.  
\end{quote}

Theorems ~\ref{t1} and ~\ref{t2}, like Durso's Theorem, use the \em entire \em spectrum.  Physically this is like having a \em perfect \em ear, which is impossible.  It is therefore interesting to consider isospectral problems involving a finite part of the spectrum.  A well known conjecture due to P\'olya and Szeg\H{o} is the following.  

\begin{conj}[P\'olya-Szeg{\H{o}}] For each $n \geq 3$, the regular $n$-gon uniquely minimizes $\lambda_1$ among all $n$-gons with fixed area. 
\end{conj} 

\begin{remark} For $n=3,4$ this is a theorem proven by P\'olya and Szeg{\H{o}} in \cite{po}.  
\end{remark} 

The physical interpretation of the P\'olya-Szeg{\H{o}} Conjecture is that the symmetry of the regular $n$-gon amongst all $n$-gons of fixed area is uniquely distinguished by its fundamental tone (corresponding to $\lambda_1$).  Our last result is known as the weak P\'olya-Szeg{\H{o}} Conjecture.

\begin{theorem}[Weak P\'olya-Szeg{\H{o}} Conjecture]  \label{t3} 
For each $n \geq 3$ there exists $N$ which depends only on $n$ such that if the first $N$ eigenvalues of a convex $n$-gon coincide with those of a regular $n$-gon, then it is congruent to that regular $n$-gon.  
\end{theorem} 

This work is intended for  a general mathematical audience, so we begin in \S~\ref{s2} with an overview of methods used in eigenvalue problems.  These are then used to prove Theorem~\ref{t1} in \S ~\ref{s3} and Theorem ~\ref{t2} as well as the Weak P\'olya-Szeg{\H{o}} Conjecture (Theorem ~\ref{t3}) in \S~\ref{s4}. Conjectures and {open problems} comprise \S~\ref{s5}.  

\section{Methods}\label{s2}
Some readers may be familiar with the so-called ``Steiner Symmetrization'' technique, named after the German mathematician Jakob Steiner \cite{stein}.  Mathematical objects are often separated into three classes, like for example conic sections:  parabolic, elliptic, and hyperbolic.  Steiner would have recognized this as an example of an old German saying ``Alle gute Dinge sind drei.''  (All good things come in three).  In this tradition we have distinguished three general types of methods used in spectral theory.  Unlike conic sections, however, these methods have non-empty intersection.

\subsection{Geometric techniques}   \label{ladj} 
Steiner symmetrization is an example of a more general technique known in this context as {\it local adjustment}.

  Quantities which are determined by the spectrum are known as \em spectral invariants, \em so one can say that these quantities can be ``heard.'' 
  We shall see in \S ~\ref{analytic} that the  the area and the perimeter of a domain can be ``heard''.
   Consequently it is possible to ``hear'' the following function 
\begin{equation} \label{f} f(\Omega) := \frac{|\Omega|}{|\pa \Omega|^2}, \end{equation}
where $|\Omega|$ denotes the area of a domain $\Omega$, and $|\pa \Omega|$ denotes the perimeter of $\Omega$.  Note that $f$ is invariant under scaling.  Therefore, $f$ detects only the shape but not the scale of a domain.  

The ancient Greeks (see \cite{iso}) proved that for any polygon $\Omega$, $f(\Omega) \leq f(\D)$, where $\D$ is a disk.  Steiner gave a beautiful and geometric proof of this fact and generalized the result to three dimensions in \cite{stein}.  We recall some of the basic ideas.  Steiner began by stating the following ``Fundamental Theorem.''  

\begin{quote}[Steiner 1838]  
\em Among all triangles with the same base and height, the isosceles has the smallest perimeter.  
\em 
\end{quote} We include a short proof by picture.  Consider Figure ~\ref{ladj1} in which the sides of the triangle $ABC$ satisfy $|BC| > |AB|$.  Moving the vertex $B$ to $B'$ such that $|A'B| = |BC|$ and $|AB'| = |A'B'| = |B'C|$, then since $A',B',A$ are collinear,  
$$|AB| + |BC| = |AB| + |A'B| > |AA'| = |AB'| + |B'C|.$$
Consequently the isosceles triangle $AB'C$ has the same area as triangle $ABC$ but smaller perimeter.  

\begin{figure} \caption{Local adjustment to equal sides increases $f$}\label{ladj1}

\includegraphics[width=100pt]{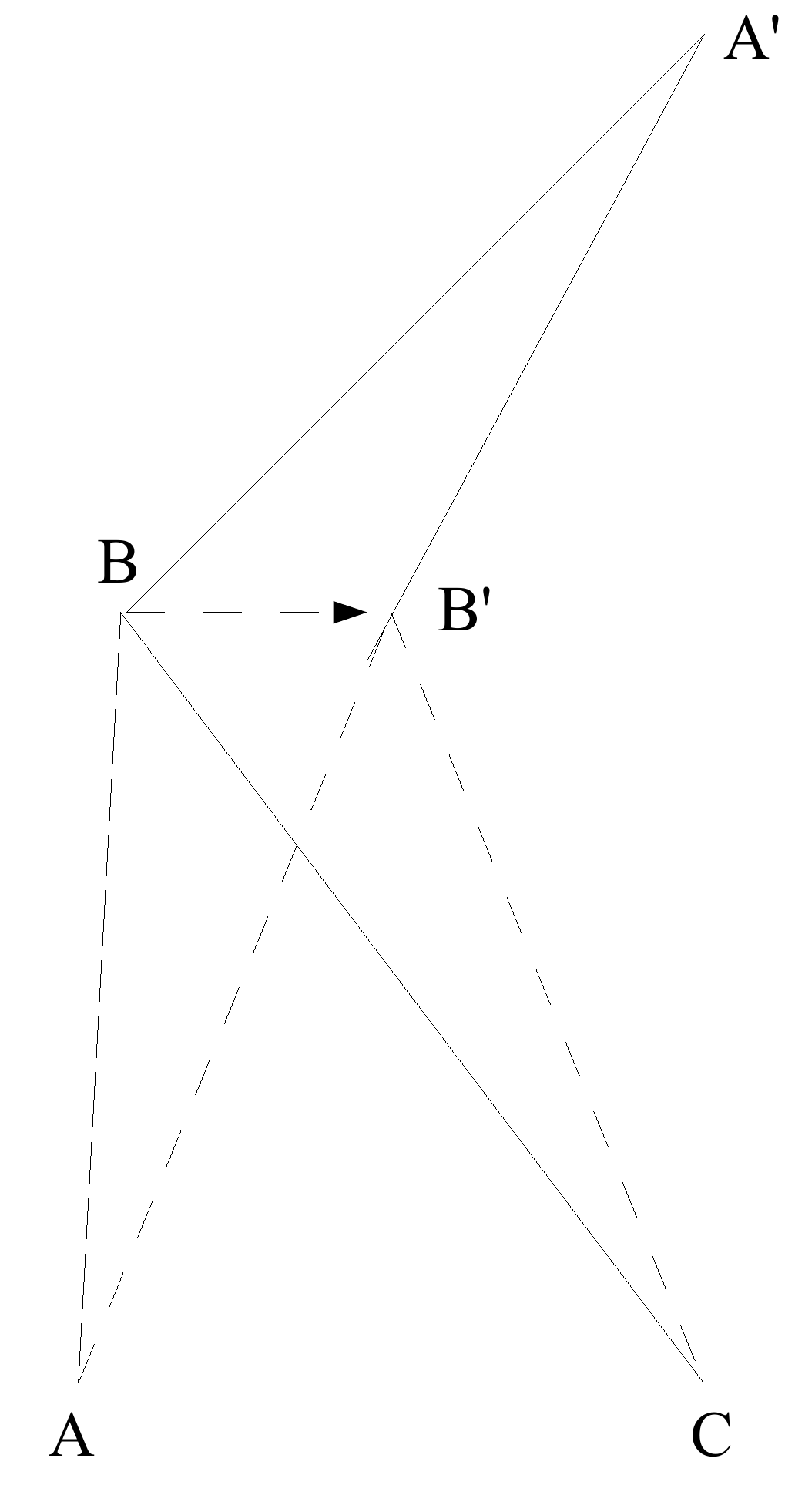} \end{figure} 

 Steiner then stated and proved the following theorem:  if the angles adjacent to at least one of the parallel sides of a trapezoid are not equal, then there is a trapezoid with the same area and base with smaller perimeter and which is symmetric about an axis.  He used these two theorems to prove a third theorem:  an arbitrary convex polygon can be deformed to a convex polygon with the same area, smaller perimeter, and which is symmetric about an axis.  This is known as ``Steiner symmetrization'' and was beautifully illustrated at the end of \cite{stein}.  This technique was used by Poly\'a and Szeg\H{o} to prove that among all convex $n$-gons of fixed area, the regular one minimizes the first eigenvalue for $n=3$, $4$.  The proof no longer works for $n \geq 5$ because Steiner's symmetrization generally increases the number of sides as can be seen in the illustrations of [ibid].  For his proof this is no problem, since his goal was to prove that $f$ is uniquely maximized by the disk.  One detail was unfortunately overlooked in Steiner's proof:  the existence of a maximizer.  Oskar Perron expanded upon and poked fun at the potential consequences of such an oversight in \cite{perron}.  
 
We shall prove a slight variation of Steiner's isoperimetric inequality which will then be used to prove Theorem 2.  
\begin{prop}\label{p1} The function $f$ defined in ~\eqref{f} is uniquely maximized among all $n$-gons by the regular $n$-gons. 
\end{prop}

\begin{proof} 
Note that reflecting across a concave vertex preserves the perimeter while increasing the area thereby increasing $f$, so we only need to consider the convex $n$-gons.  

Our proof is related to Steiner's but differs slightly due to the fact that our adjustment techniques should preserve the number of sides.  The proof is by induction on the number of sides.    We first address the existence of a maximizer.  For the regular $n$-gon $R_n$, 
$$f(R_n) = \frac{1}{4n \tan (\pi/n)} < \frac{1}{4 \pi} = f(\D).$$
We see that $\{f(R_n)\}_{n=1} ^\infty$ is a strictly increasing sequence.  Let the set of all convex $n$-gons be denoted by $\M_n$.  Since $f$ is bounded above, let's define $f_\infty$ to be the supremum of $f$ over all $n$-gons.  Since it's the supremum, there must be a sequence of $n$-gons $\{P_k\}_{k=1} ^\infty$ with diameter equal to a fixed constant such that $f(P_k) \to f_\infty$.  What can happen to this sequence as $k \to \infty$?  If the $P_k$'s were to collapse to a segment, then the area would tend to $0$ but the perimeter would stay bounded from below and consequently 
$$\lim_{k \to \infty} f(P_k) = 0 < f(R_n).$$
For the base case of induction, $n=3$, and so we see that the sequence cannot collapse, and since the diameters are all equal to a fixed constant, a subsequence of $P_k$ converge to some triangle $T$. Clearly $f$ is a continuous function  on $\M_3$, and so $f_\infty = f(T)$ in this case.  By Steiner's Fundamental Theorem $T$ must be equilateral;  the details of this proof are left to the reader.  

Proceeding by induction we assume that for some $N \geq 3$ the proposition is true for all $n \leq N$.  Let $n = N+1$.  If the sequence of $n$-gons $P_k$ such that 
$$f(P_k) \to \sup \{ f(\Omega) : \Omega \in \M_n \}, \quad \textrm{ as } k \to \infty$$ collapses to an $m$-gon $P$ for some $m<n$, then by the induction hypothesis 
$$f(P_k) \to f(P) \leq f(R_m) < f(R_n), \quad \textrm{where $R_m$ is a regular $m$-gon.} $$  
This contradicts the definition of the sequence $\{P_k \}_{k=1} ^\infty$.  Consequently, restricting to a subsequence if necessary, we may assume that $P_k \to P \in \M_n$.  

By Steiner's Fundamental Theorem adjacent sides of $P$ must be equal (see Figure ~\ref{ladj1}).  

We shall use local adjustment to show that $P$ must also be equiangular.  Since $f$ is scale invariant and $P$ has equal sides, let's assume that the sides of $P$ all have length equal to one.  Consider the edge between two vertices $v_{i}$ and $v_{i+1}$, where the vertices are considered modulo $n$, that is 
$$n+k \equiv k \textrm{ mod } n, \quad 1 \leq k \leq n.$$
Denote by $\{\gamma_i \}_{i=1} ^n$ the set of interior angles of $P$.  The exterior angles, 
$$\alpha_i := \pi - \gamma_i, \quad i=1, \ldots, n.$$
Let's think about what happens if we move the edge between $v_i$ and $v_{i+1}$ in the direction of the outward normal to this edge, in other words moving the edge parallel to its position in the direction away from the interior of $P$; see Figure ~\ref{ladj2}.  Let $t$ denote the distance the edge is translated.  Then the area 
$$A(t) = A + t + O(t^2),$$
where $A$ is the area of $P$.  The perimeter 
$$L(t) = L + t \csc \alpha_i - t \cot \alpha_i + t \csc \alpha_{i+1} - t \cot \alpha_{i+1} + O(t^2),$$
where $L$ is the perimeter of $P$.  We therefore define  
$$\vphi(x) := \csc x - \cot x = \tan \left( \frac{x}{2} \right).$$

\begin{figure}\caption{Local adjustment to equal angles increases $f$} \label{ladj2} \includegraphics[width=200pt]{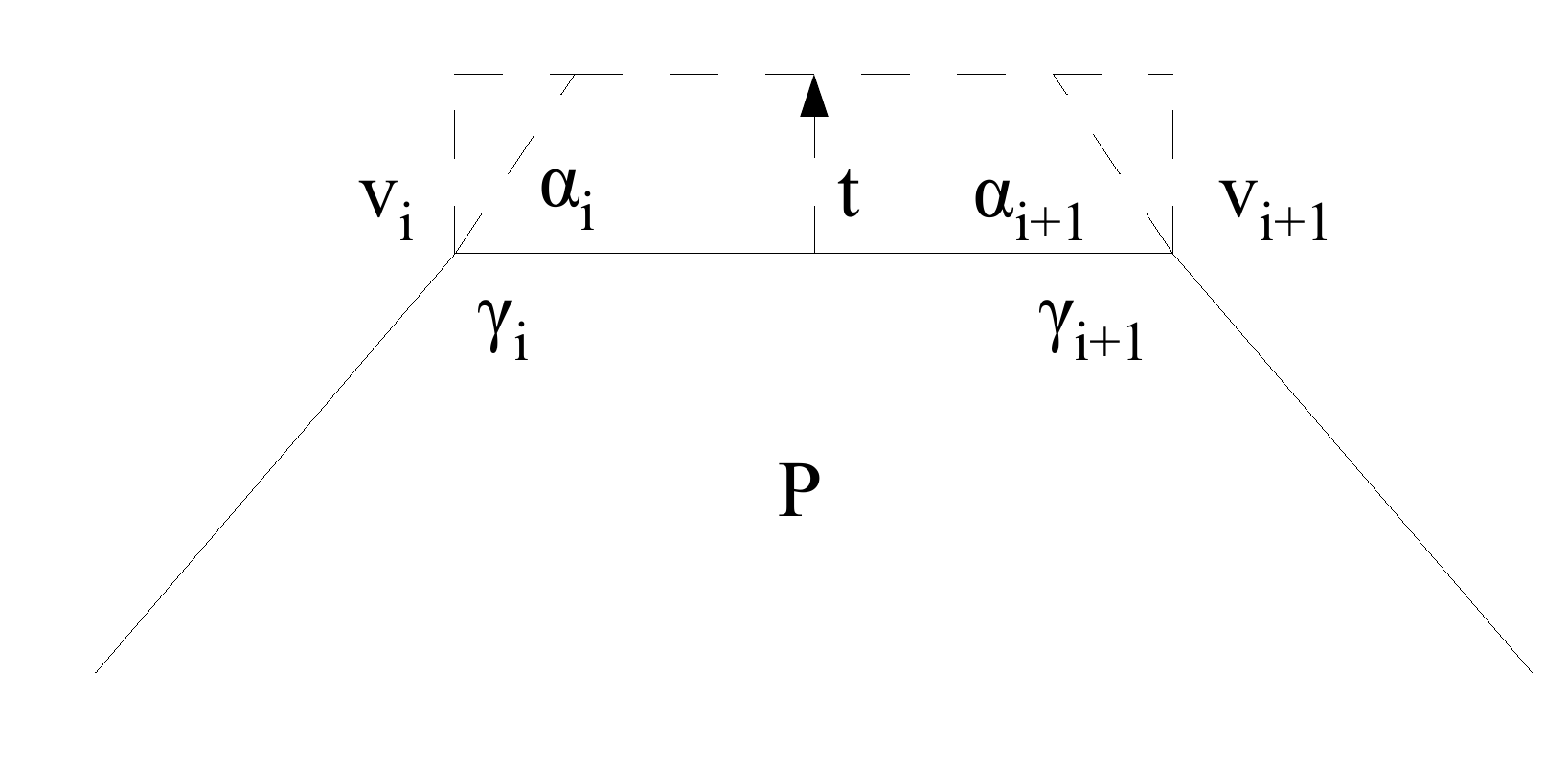} \end{figure}

The function $f$ can be considered as a function of the parameter $t$, where $t=0$ corresponds to $P$.    Since $P$ maximizes the function $f$, 
$$f'(0) = f'(P) = 0.$$
Substituting the expressions for $A(t)$ and $L(t)$ above, 
$$f(t) = \frac{A(t)}{L(t)^2} = \frac{A + t + O(t^2)}{(L + t (\vphi(\alpha_i) + \vphi(\alpha_{i+1}))+ O(t^2))^2} $$
$$= \left( \frac{A}{L^2} + \frac{t}{L^2} + O(t^2) \right) \left( 1 + \frac{2t}{L} \left( \vphi(\alpha_i) + \vphi(\alpha_{i+1}) \right) + O(t^2) \right)^{-1}.$$
For $t$ sufficiently small we can expand the last term on the right in a geometric series, 
$$f(t) = \frac{A}{L^2} + t \left( \frac{1}{L^2} - \frac{2A}{L^3} \left(\vphi(\alpha_i) + \vphi(\alpha_{i+1}) \right) \right) + O(t^2).$$
By calculus, the derivative of $f$ at $t=0$ is the coefficient of $t$, so 
$$f'(0) =  \frac{1}{L^2} - \frac{2A}{L^3} \left(\vphi(\alpha_i) + \vphi(\alpha_{i+1}) \right) = 0 \implies \vphi(\alpha_i) + \vphi(\alpha_{i+1}) = \frac{L}{2A}.$$
Since $1\leq i \leq n$ was chosen arbitrarily, this shows that 
\begin{equation} \label{vphi1} \vphi(\alpha_i) + \vphi(\alpha_{i+1}) = \frac{L}{2A}, \quad \forall i \textrm{ modulo } n. \end{equation} 
This implies 
$$\vphi(\alpha_i) = \vphi(\alpha_{i+2}) \quad \forall i \textrm{ modulo } n.$$
Since $\vphi(x) = \tan (x/2)$ is a monotonically increasing function on $(0,\pi)$, and the interior angles $\alpha_i \in (0, \pi)$, $\vphi$ is injective and 
\begin{equation} \label{modn} \alpha_i = \alpha_{i+2} \quad \forall i \textrm{ modulo } n. \end{equation} 
If $n$ is odd, then we're done because 
$$\alpha_1 = \alpha_3 = \ldots = \alpha_n = \alpha_{n+2} = \alpha_2  = \alpha_4 = \ldots = \alpha_{n-1},$$
which shows that $P$ is equiangular.  If $n$ is even, let's take a closer look at $\vphi$.  Since
$$\sum_{i=1} ^n \alpha_i = 2\pi,$$ 
by ~\eqref{modn}
$$\frac{n}{2} \alpha_1 + \frac{n}{2} \alpha_2 = 2\pi \implies \alpha_1 + \alpha_2 = \frac{4\pi}{n}.$$
Consequently
\begin{equation} \label{vphi2} \vphi(\alpha_1) + \vphi(\alpha_2) = \vphi(\alpha_1) + \vphi\left( \frac{4\pi}{n} - \alpha_1 \right) = \tan\left(\frac{\alpha_1}{2} \right) + \tan \left( \frac{2\pi}{n} - \frac{\alpha_1}{2} \right).\end{equation}  
By switching names if necessary we may assume without loss of generality that $0 < \alpha_1 \leq \frac{2\pi}{n}$.  Since the sides all have length $1$, it follows that $L=n$.  Putting ~\eqref{vphi1} together with ~\eqref{vphi2} shows that 
$$ \vphi(\alpha_1) + \vphi(\alpha_2) =  \tan\left(\frac{\alpha_1}{2} \right) + \tan \left( \frac{2\pi}{n} - \frac{\alpha_1}{2} \right) = \frac{n}{2A},$$
and 
$$\frac{1}{f(P)} = \frac{L^2}{A} = 2n \left( \vphi(\alpha_1) + \vphi(\alpha_2) \right)  = 2n \left( \tan\left(\frac{\alpha_1}{2} \right) + \tan \left( \frac{2\pi}{n} - \frac{\alpha_1}{2} \right) \right).$$
Since $P$ maximizes $f$, it follows that $P$ minimizes $1/f$.  The derivative with respect to $\alpha_1$ of the function
$$\frac{1}{f} = 2n \left( \tan\left(\frac{\alpha_1}{2} \right) + \tan \left( \frac{2\pi}{n} - \frac{\alpha_1}{2} \right) \right),$$
is strictly negative for $\alpha_1 \in (0, 2\pi/n)$.  Since $P$ minimizes $1/f$, 
$$\alpha_1 = \frac{2\pi}{n} \implies \alpha_2 = \frac{4\pi}{n} - \alpha_1 = \frac{2\pi}{n} \implies \alpha_i = \frac{2\pi}{n} \quad \forall i.$$
\end{proof}

\subsection{Analytic techniques}  \label{analytic} 
Using the definition of the eigenvalues and the chain rule from calculus, one can prove the following scaling property of the eigenvalues 
$$\Omega \mapsto c \Omega \implies \lambda_k (\Omega) \mapsto c^{-2} \lambda_k (\Omega).$$
This simply means that a domain $\Omega$ is scaled by a constant factor $c$, so if $\Omega$ is a polygon, then its side lengths are all multiplied by $c$, and the eigenvalues are divided by $c^2$.  

\subsubsection{Variational principles} 
The eigenvalues can be defined by a \em variational principle, \em also known as a ``mini-max'' principle.  The eigenvalues are the infima of the Rayleigh-Ritz quotient 
\begin{equation} \label{rk} \lambda_k = \inf \left\{ \left . \frac{ \int_\Omega |\nabla f |^2 dxdy }{\int_\Omega f^2 dx dy} \right|,  \int_{\Omega} f f_j dx dy =0,  j=0, \ldots, k-1 \right\} \end{equation}
where $f_0 \equiv 0$, and $f_j$ is an eigenfunction for $\lambda_j$ for $j\geq 1$.  These formulae can be found in \cite{chavel}.  An equivalent formula found in \cite{cour-hil} is the so-called ``maxi-min'' principle 
\begin{equation} \label{maximin} \lambda_k = \inf_{L \subset H^{1,2} _0 (\Omega), \, \dim(L) = k} \left\{ \sup_{f \in L} \frac{ \int_\Omega |\nabla f |^2 dxdy }{\int_\Omega f^2dxdy }  \right\}. \end{equation} 
In both variational formulae the Rayleigh-Ritz quotient is taken over all $\cC^2$ smooth functions on $\Omega$ which vanish on the boundary and are not identically zero.  The variational principles can be used that the eigenvalues are continuous functions of the domains, so if a sequence of domains $\Omega_k \to \Omega_0$ as $k \to \infty$, then the eigenvalues 
$$\lambda_i (\Omega_k) \to \lambda_i (\Omega_0), \quad \textrm{ as } k \to \infty, \textrm{ for each } i \in \N.$$ 
The maxi-min principle can be used to prove \em domain monotonicity:  \em 
$$\Omega \subseteq \Omega' \implies \lambda_i (\Omega) \geq \lambda_i (\Omega'), \quad \forall i.$$
For a one dimensional domain, this simply means that shorter strings produce higher frequencies, a fact well-known by musicians. 

\subsubsection{Fundamental gap}  
The difference between the first two eigenvalues is known as the \em fundamental gap.  \em  A recent theorem proven by Andrews and Clutterbuck \cite{ac} shows that if the diameter of a convex domain tends to zero, then its fundamental gap blows up.  On the other hand, we'll see in the proposition below that if a domain is very large, then its fundamental gap becomes small.  The proposition will be a key ingredient in the proof of the Weak P\'olya-Szeg\H{o} conjecture. 

\begin{prop}  \label{pr-a} If a sequence of convex domains $\{\Omega_k\}_{k=1} ^\infty$ in $\R^2$ satisfies 
$$\lambda_i (\Omega_k) = \lambda_i (\Omega_j), \quad \forall j, k \in \N, \quad i=1,2,$$
then both the diameters and the in-radii of the domains are contained in a compact subset of $(0, \infty)$.  
\end{prop} 
\begin{proof} If the in-radii of the domains $r_k \to \infty$, then the domains contain bigger and bigger disks, and so we can estimate using domain monotonicity because there are formulae for the eigenvalues of the disk. 

\begin{ex} Prove that the eigenvalues of the disk of radius $R$ are 
$$R^{-2} j_{m,n} ^2, \quad m \in \N, \quad n \in \{0\} \cup \N,$$
where $j_{m,n}$ is the $m^{th}$ zero of the Bessel function $J_n$ of order $n$. 
\end{ex} 

By the exercise and domain monotonicity
$$\lambda_1 (\Omega_k) \leq \frac{1}{r_k^2}  j_{1,0} ^2 \to 0, \textrm{ as } r_k \to \infty.$$
Since the first eigenvalue is fixed, this is impossible.  

On the other hand, if the in-radii tend to $0$, then there are rectangles $R_k$ of height $h_k$ such that $\Omega_k \subset R_k$ and $h_k \to 0$.  By domain monotonicity, 
$$\lambda_1 (\Omega_k) \geq \lambda_1 (R_k) > \frac{\pi^2}{h_k^2} \to \infty, \quad \textrm{ as } k \to \infty,$$
which is impossible, and since the diameter is bounded below by the in-radius,  the diameters also cannot tend to $0$

What happens if the in-radii stay bounded away from $0$ and $\infty$ while the diameters $d_k$  tend toward $\infty$?  Let's define the \em width \em $w_k$ of $\Omega_k$ to be the shortest distance between two infinite parallel lines such that $\Omega_k$ fits within a strip of width $w_k$.  The in-radii are bounded away from both $0$ and $\infty$ and so the widths are also bounded uniformly away from $0$ and $\infty$.  Next, let's rotate and translate the domains such that the width $w_k$ goes from the point $A:=(0, w_k/2)$ to the point $B:=(0, -w_k/2)$, and these points lie on the boundary of $\Omega_k$.  By domain monotonicity, since $\Omega_k$ is contained in a rectangle of width $w_k$, 
$$\lambda_1 (\Omega_k) \geq \lambda_1 (\textrm{Rectangle of width $w_k$}) > \frac{\pi^2}{w_k^2}.$$
Reflecting across the vertical axis if necessary, we may assume there is a point in $\Omega_k$ whose horizontal component is $d_k/3$.  Let's call this point $C$; see Figure ~\ref{f-tri}.  

\begin{figure} \caption{Triangles $T_k$ and $T_k'$}\label{f-tri}

\includegraphics[width=160pt]{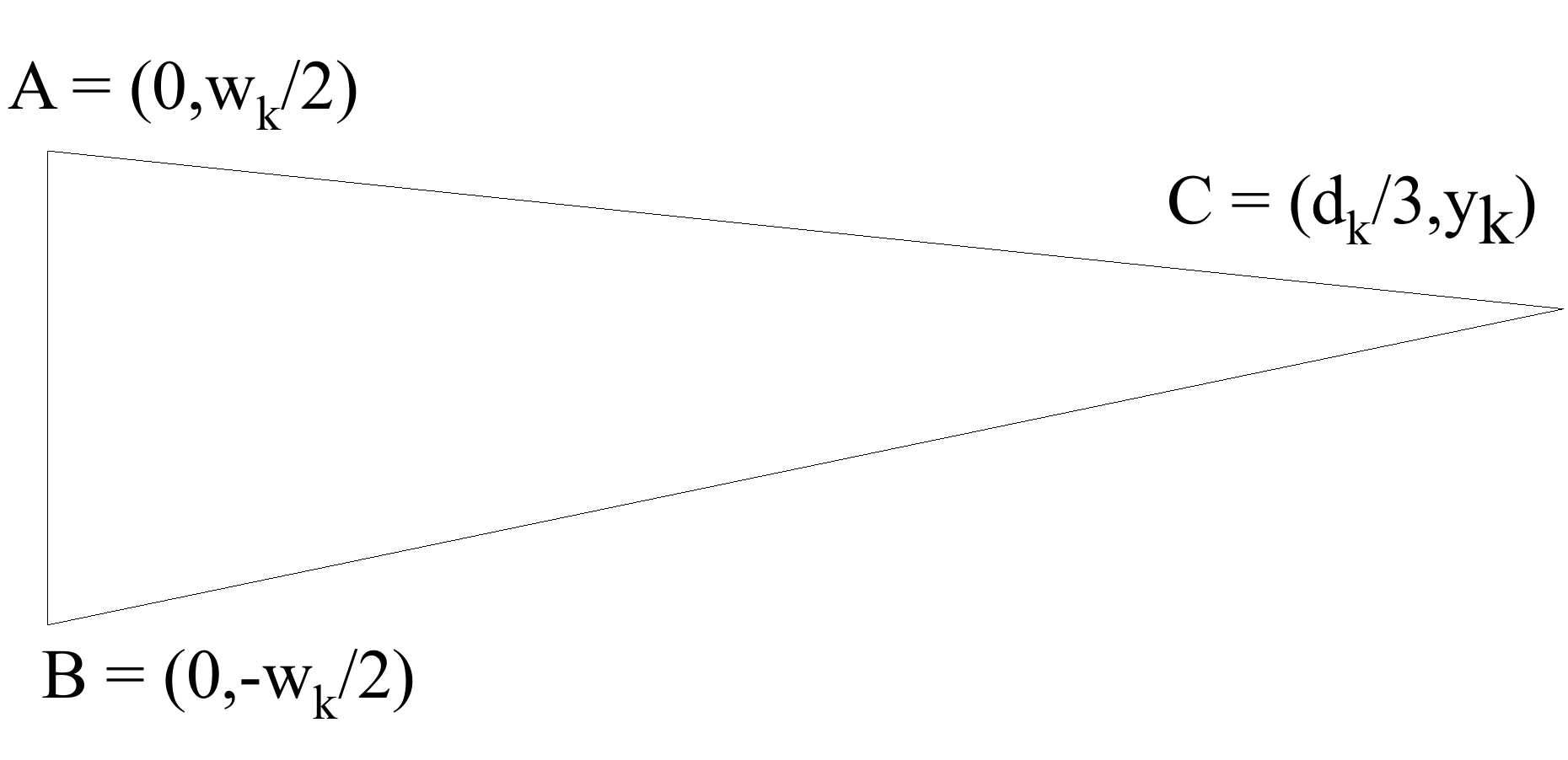} \includegraphics[width=110pt]{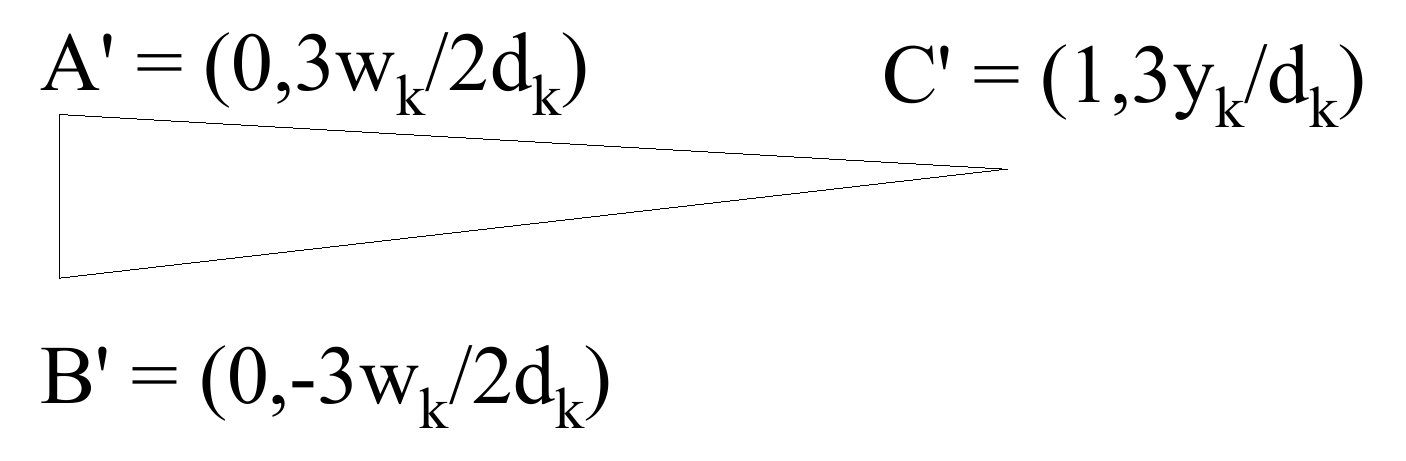} \end{figure} 

By convexity, triangle $T_k :=ABC$ is contained in $\Omega_k$.  It follows from \cite{crm} (c.f. \cite{strip}) that 
$$\lambda_2 (T_k) \leq \frac{\pi^2}{w_k^2} + c d_k^{-2/3},$$
for a fixed constant $c>0$.  To see this, we scale $T_k$ by $\frac{3}{d_k}$.  The resulting triangle, $T_k'$ has one angle that is very small and tends to zero as $d_k \to \infty$, while the other two angles remain bounded away from zero, and their opposite sides tend toward $1$.  By the estimates in the proof of Proposition 1 of \cite{crm} (c.f. similar estimates in \cite{frtri}) 
$$\lambda_2 (T_k') \leq \frac{\pi^2 d_k^2}{9 w_k^2} + c' d_k^{4/3},$$
for a fixed constant $c'>0$.  Rescaling by $\frac{d_k}{3}$ and using the scaling property of the eigenvalues, 
$$\lambda_2 (T_k)  = \frac{9}{d_k^2} \lambda_1 (T_k') \leq \frac{\pi^2}{w_k^2} + cd_k^{-2/3}, \quad c = 9c'.$$
Consequently, by domain monotonicity since $T_k \subset \Omega_k$, 
$$\lambda_2 (\Omega_k) \leq \lambda_2 (T_k) \leq \frac{\pi^2}{w_k^2} + c d_k^{-2/3}.$$
Together with the lower bound for $\lambda_1 (\Omega_k)$ we have
$$\lambda_2 (\Omega_k) - \lambda_1 (\Omega_k) \leq c d_k^{-2/3} \to 0 \textrm{ as } d_k \to \infty.$$
Since these eigenvalues are fixed, their difference can't vanish as $k \to \infty$, which shows that both the diameters and the in-radii must be contained in a compact subset of $(0, \infty)$.  
\end{proof} 

\begin{remark} Although many authors have studied lower bounds for the fundamental gap, upper bounds appear to be scarce in the literature and may warrant further investigation.  The proposition is a tiny bit of progress in this direction because the proof implies that if a sequence of convex domains has in-radii bounded away from both zero and infinity, then there is a constant $c>0$ such that the fundamental gap 
$$\lambda_2 (\Omega_k) - \lambda_1 (\Omega_k) \leq c d_k ^{-2/3}.$$
\end{remark}

\subsubsection{The heat trace}  \label{s-heat} 
The spectrum not only determines the resonant frequencies of vibration but also the flow of heat.  The heat trace is 
$$\sum_{k=1} ^\infty e^{-\lambda_k t}.$$  
The eigenvalues $\{\lambda_k \}_{k=1} ^\infty$ grow at an asymptotic rate known as \em Weyl's law \em \cite{weyl} which can be used to prove that the sum converges uniformly for all $t \geq t_0 > 0$ for any positive $t_0$ but diverges as $t \to 0$. Before Kac's article was published, mathematicians and physicists had observed that the way in which the heat trace diverges for $t \downarrow 0$ can be used to show that the spectrum determines certain geometric features.   Pleijel \cite{pl} proved that the heat trace admits an asymptotic expansion as $t \to 0$ of  the form 
$$ \sum_{k=1} ^\infty e^{-\lambda_k t} \sim \frac{|\Omega|}{4 \pi t} - \frac{| \pa \Omega|}{8 \sqrt{\pi t}}, \quad t \downarrow 0, $$
where $|\Omega|$ denotes the area of $\Omega$, and $|\pa \Omega|$ is the perimeter.  Kac determined the third term in the asymptotic expansion in \cite{kac} which we briefly recall the key ideas.  

The heat kernel on $\R^2$ can be explicitly computed to be 
$$H_E (x,y, x', y', t) = \frac{1}{4 \pi t} e^{-|(x,y) - (x', y')|^2 / 4t}.$$
Integrating along the diagonal $x=x', y=y'$ over a bounded region $R \subset \R^2$ gives 
$$\int_R \frac{1}{4 \pi t} dx dy = \frac{|R|}{4 \pi t}.$$
At each interior point there is a neighborhood which does not intersect the boundary, and on this neighborhood the heat kernel for $\Omega$ is ``close'' to the Euclidean heat kernel $H_E$ for short times.  Kac referred to this as ``not feeling the boundary'' \cite{kac}.  He showed that the heat trace for a planar domain is asymptotic to the trace over the domain of the Euclidean heat kernel, 
$$\sum_{k=1} ^\infty e^{-\lambda_k t} \sim \int_{\Omega} H_E (x, y, x, y, t) dx dy = \frac{|\Omega|}{4 \pi t}.$$
To compute the next term in the asymptotic expansion he considered the behavior at the boundary both near and away from the corners.  The principle of ``not feeling the boundary'' turns out to be a general phenomenon which we call a ``locality principle.''  The locality principle is that if one cuts a domain into neighborhoods and uses a model heat kernel for each neighborhood (with the correct boundary condition if the neighborhood intersects the boundary), then the heat kernel for the domain $\Omega$ integrated over each neighborhood is asymptotically equal as $t \downarrow 0$ to the model heat kernel integrated over the same neighborhood.  More precisely, the heat trace for a locally constructed parametrix is asymptotically equal to the actual heat trace over $\Omega$ as $t \downarrow 0$.  In the case of a polygonal domain there are three types of neighborhoods and three model heat kernels:  the Euclidean heat kernel for interior neighborhoods, the heat kernel for a half plane for edge neighborhoods away from the vertices, and the heat kernel for a circular sector of opening angle equal to that of the opening angle at a vertex.  Kac used the explicit formulae for these model heat kernels and the locality principle to show\footnote{Kac did not compute the closed formula we have here, which is due to Dan Ray (unpublished) and Fedosov (in Russian) \cite{fed} and appeared in \cite{ms}; a particularly transparent proof is in \cite{vdb}.} that for a convex $n$-gon $\Omega$ with interior angles $\{\alpha_i\}_{i=1} ^n$, 
\begin{equation} \label{heat} \sum_{k=1} ^\infty e^{-\lambda_k t} \sim \frac{|\Omega|}{4 \pi t} - \frac{| \pa \Omega|}{8 \sqrt{\pi t}} + \sum_{i=1} ^n \frac{\pi^2 - \alpha_i ^2}{24 \pi \alpha_i}  \quad t \downarrow 0. \end{equation} 
This shows that the area, the perimeter, and the sum over the angles of $\frac{\pi^2 - \alpha^2}{24 \pi \alpha}$ are all spectral invariants.   

\subsubsection{The wave trace} 
The spectrum also determines the \em wave trace.  \em  Imagine a convex $n$-gon is a billiard table.  The set of closed geodesics is precisely the set of all paths along which a billiard ball hit with a pool cue could roll, such that the ball returns to its starting point.  The set of lengths of closed geodesics, which is known as \em the length spectrum, \em is related to the (Laplace) spectrum by a deep result proven by  Duistermaat and Guillemin \cite{dg} in the late 1970s.  The wave trace is a tempered distribution defined by  
$$\sum_{k=1} ^\infty e^{i \sqrt{\lambda_k} t}.$$
Duistermaat and Guillemin proved in [ibid] that the wave trace has singularities precisely at times t equal to the lengths of closed geodesics.  Although their result was for closed Riemannian manifolds, it holds analogously for polygonal domains as shown by F.~G.~Friedlander \cite{fried}.  If we know the entire spectrum $\{\lambda_k\}_{k=1} ^\infty$, then we know the wave trace and therefore the times at which it is singular.  This means that the (Laplace) spectrum determines the length spectrum.  It follows that the set of lengths of closed geodesics is a spectral invariant. As we shall see in \S~\ref{s3}, it can be significantly more complicated to extract geometric information from the wave trace as compared with the heat trace.  For this reason the wave trace is often considered a more subtle spectral invariant than the heat trace. 


\subsection{Algebraic techniques}  
Shortly after Durso completed her Ph.~D. thesis \cite{dur}, Chang and De~Turck published the following isospectrality result for triangular domains.  

\begin{theorem}[Chang and De Turck \cite{cd}] Let $T_0$ be a Euclidean triangle.  There is an integer $N$ which depends only on the first two eigenvalues of $T_0$ such that if $T_1$ is another triangle whose first $N$ eigenvalues coincide with those of $T_0$, then all the eigenvalues coincide.  
\end{theorem}

 Triangular domains can be parametrized to depend analytically on three parameters.  The above theorem is an application of the following more general result  for families of Riemannian metrics depending analytically on finitely many parameters.  

\begin{theorem}[Chang and De Turck \cite{cd}]\label{th-cd}  Let $D$ be a compact oriented manifold (with or without boundary) of dimension $n$.  We consider a family of metrics $g(\eps)$, depending analytically on the parameter $\eps \in \R^p$.  Let $\lambda_k (\eps)$ denote the $k^{th}$ eigenvalue of the Laplacian of $g(\eps)$ on $D$, and if $\pa D \neq \emptyset$ we assume that $\pa D$ is piecewise smooth and we impose the Dirichlet boundary condition.  We also let $\sigma(\eps)$ denote the spectrum of the Laplacian of $g(\eps)$.  Under these assumptions, for each compact subset $K \subset \R^p$ there is an integer $N = N(K)$ such that if $\eps_0$, $\eps_1 \in K$ and $\lambda_j (\eps_0) = \lambda_j (\eps_1)$ for all $j=1, \ldots, N$ then $\sigma(\eps_0) = \sigma(\eps_1)$.  In other words, for $\eps \in K$ the entire spectrum of the Laplacian of $g(\eps)$ is determined by the values of the first $N(K)$ eigenvalues.  
\end{theorem}

From \cite{cd} we quote, ``The key ingredients in the proof are the assumption of real-analytic dependence of the metric on the parameters, and the resulting real-analytic dependence of certain symmetric functions of the eigenvalues, and finally the fact that the ring of germs of real analytic functions of finitely many variables is Noetherian.'' 

In order to use the above result, we shall prove the following.  

\begin{prop}\label{propcd}  For each $n\geq 3$ let $\M_n$ denote the set of all convex $n$-gons.  Then $\M_n$ can be identified with a family of Riemannian metrics on the unit disk which depend real analytically on finitely many parameters.  \end{prop}
 
\begin{proof}
Let $\D$ denote the unit disk.    One of the most beautiful theorems in complex analysis is the Uniformization Theorem which states that all simply connected bounded domains in the plane are conformally equivalent to the disk.  For polygonal domains, there is an explicit formula for the conformal map known as the Schwarz-Christoffel formula.  Let 
\begin{equation} \label{scmap} f(z) := \int_0 ^z \prod_{i=1} ^n (w-w_i)^{\alpha_i - \pi} dw: z \in \D \to f(z) \in P.\end{equation} 
This function is a conformal map from the disk $\D$ to the polygon $P$ with interior angles $\{\alpha_i\}_{i=1}^n$ and vertices $p_i = f(w_i)$, where the points $w_i$ lie on the boundary of the unit disk. Let us fix points $w_1$ and $w_2$ in $\pa \D$ such that the length of the shortest side of $P$ is $|p_1 - p_2| = |f(w_1) - f(w_2)| = 1$.  By Theorem 3.1 of \cite{sc}, the $n-1$ angles $\{ \alpha_i \}_{i=1} ^{n-1}$ together with the $n-3$ side lengths 
$$|p_i - p_{i+1}| , \quad i=2, \ldots , n-1,$$
uniquely determine $P$ under the assumption that the shortest side length of $P$ is equal to one.  Moreover, since the points $p_1$ and $p_2$ and the angles are fixed, the side lengths $|p_i - p_{i+1}|$ uniquely determine the location of the points $w_3$, $\ldots$, $w_n \in \pa \D$.  We therefore define   
$$f_\epsilon (z) : = c f(z), \quad \epsilon := (\alpha_1, \ldots, \alpha_{n-1}, p_3, \ldots, p_n, c) \in \R^{2n-3}.$$
This function is holomorphic in $\D$, piecewise holomorphic on $\pa \D$, and continuous up to $\pa \D$.  
We consider the family of metrics $g(\epsilon)$ on $\D$, where $g(\epsilon)$ is the pull-back of the Euclidean metric on $P$ with respect to the function $f_\eps$.  Then $P$ with the standard Euclidean metric is equivalent to $\D$ with the metric $g(\epsilon)$, and the spectrum of the Euclidean Laplacian on $P$ with Dirichlet boundary condition is identical to the spectrum of the Laplacian with respect to the metric $g(\epsilon)$ on $\D$ with Dirichlet boundary condition.  
\end{proof}

\section{Hearing quadrilaterals}\label{s3}
With the techniques introduced in the last section, we shall first prove that one can hear the shape of a parallelogram.  

\begin{proof}[Proof of hearing a parallelogram]
The length of the parallelogram is the length of its longer side, $L$, and the width $W$ is the length of the adjacent side.  So, 
$$W \leq L,$$
and the perimeter 
$$P = 2(L+W).$$
If the height of the parallelogram is $h$, then the area 
$$A = Lh.$$
The parallelogram has four interior angles, and the smallest has measure $\alpha \leq \frac{\pi}{2}$.  The other angles have measure $\pi - \alpha$.  So, the constant term in the short time asymptotic expansion of the heat trace ~\eqref{heat} is 
$$a_0 = \frac{\pi^2}{12 \alpha (\pi - \alpha)} - \frac{1}{12}.$$
Let's consider the function 
$$f(\alpha) = \frac{1}{\alpha (\pi - \alpha)},$$
and its derivative 
$$f'(\alpha) = \frac{2\alpha - \pi}{(\pi \alpha - \alpha^2)^2} = 0 \iff \alpha = \frac{\pi}{2}.$$
So we see that $f$ is injective on $(0, \pi/2)$ which shows that the angle $\alpha$ is uniquely determined by $a_0$ which in turn is uniquely determined by the spectrum.  

\begin{figure} \caption{A parallelogram} \label{par} \includegraphics[width=200pt]{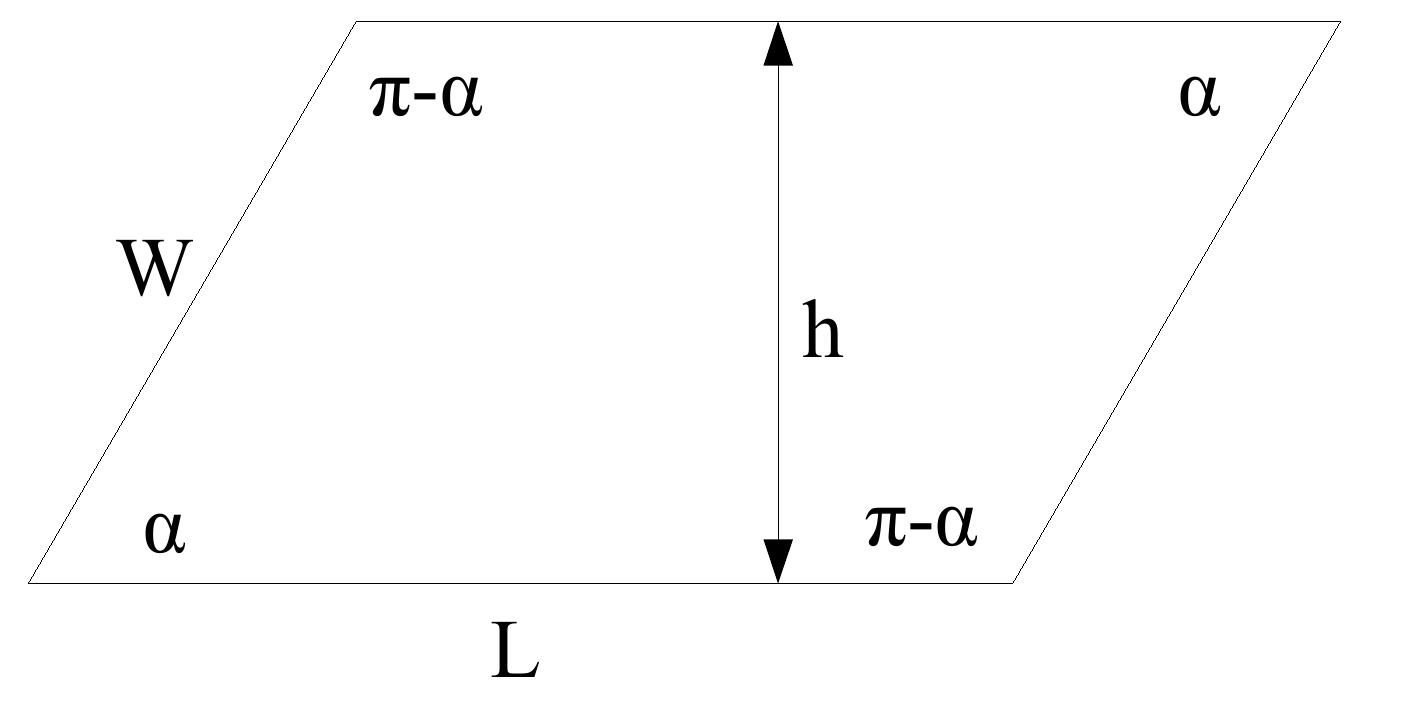} \end{figure} 

By elementary geometry, 
\begin{equation} \label{whl} W = \frac{h}{ \sin \alpha} \implies L = \frac{P}{2} - \frac{h}{ \sin \alpha},\quad A = h \left(\frac{P}{2} - \frac{h}{ \sin \alpha} \right). \end{equation} 
This is a quadratic expression for $h$, so we can use the quadratic formula to solve for $h$
$$h = \frac{P\sin\alpha}{4} \pm \sin\alpha \frac{\sqrt{\frac{P^2}{4} - \frac{4A}{\sin \alpha}}}{2}.$$
Well, there are two possible solutions, so to determine which is correct, let's think about the length and the width.  Multiplying the equation 
$$L = \frac{P}{2} - \frac{h}{\sin \alpha}$$
by $\frac{\sin\alpha}{2}$, 
$$\frac{L \sin \alpha}{2} = \frac{P \sin \alpha}{4} - \frac{h}{2} \implies \frac{P \sin \alpha}{4} = \frac{h}{2} + \frac{L\sin\alpha}{2} \geq h.$$
This shows that the only solution consistent with the geometry is 
$$h = \frac{P\sin\alpha}{4} - \sin\alpha \frac{\sqrt{\frac{P^2}{4} - \frac{4A}{\sin \alpha}}}{2}.$$
The spectrum uniquely determines the heat trace which determines $P$, $A$, $a_0$, and $\alpha$, and these uniquely determine $h$.  By ~\eqref{whl}, $h$ and $\alpha$ uniquely determine $W$ and $L$, so we see that if two parallelograms have identical spectrum, then they are congruent.  
\end{proof} 

 \begin{remark}  Durso proved that isospectral triangles are congruent using both  the heat   and the wave traces~\cite{dur}.  Grieser and Maronna gave a more elementary proof using only the heat trace in \cite{gm}.  
\end{remark} 

\begin{ex} Use the first two terms in ~\eqref{heat} together with the fact that the spectrum determines the wave trace and hence the length of the shortest closed geodesic to give a shorter proof that one can hear the shape of a parallelogram.  
\end{ex}  

\begin{defn}\label{dtrap}  An \em acute trapezoid \em is a convex quadrilateral which has two parallel sides of lengths $b$ and $B$ with $B \geq b$, and two non-parallel sides known as \em legs \em of lengths $\ell$ and $\ell'$.  The side of length $B$ is the \em base, \em and the angles at the base, $\alpha$ and $\beta$ satisfy 
$$\alpha + \beta < \frac{\pi}{2}.$$ 
\end{defn}

\begin{figure}\caption{An acute trapezoid} \label{trap} \includegraphics[width=200pt]{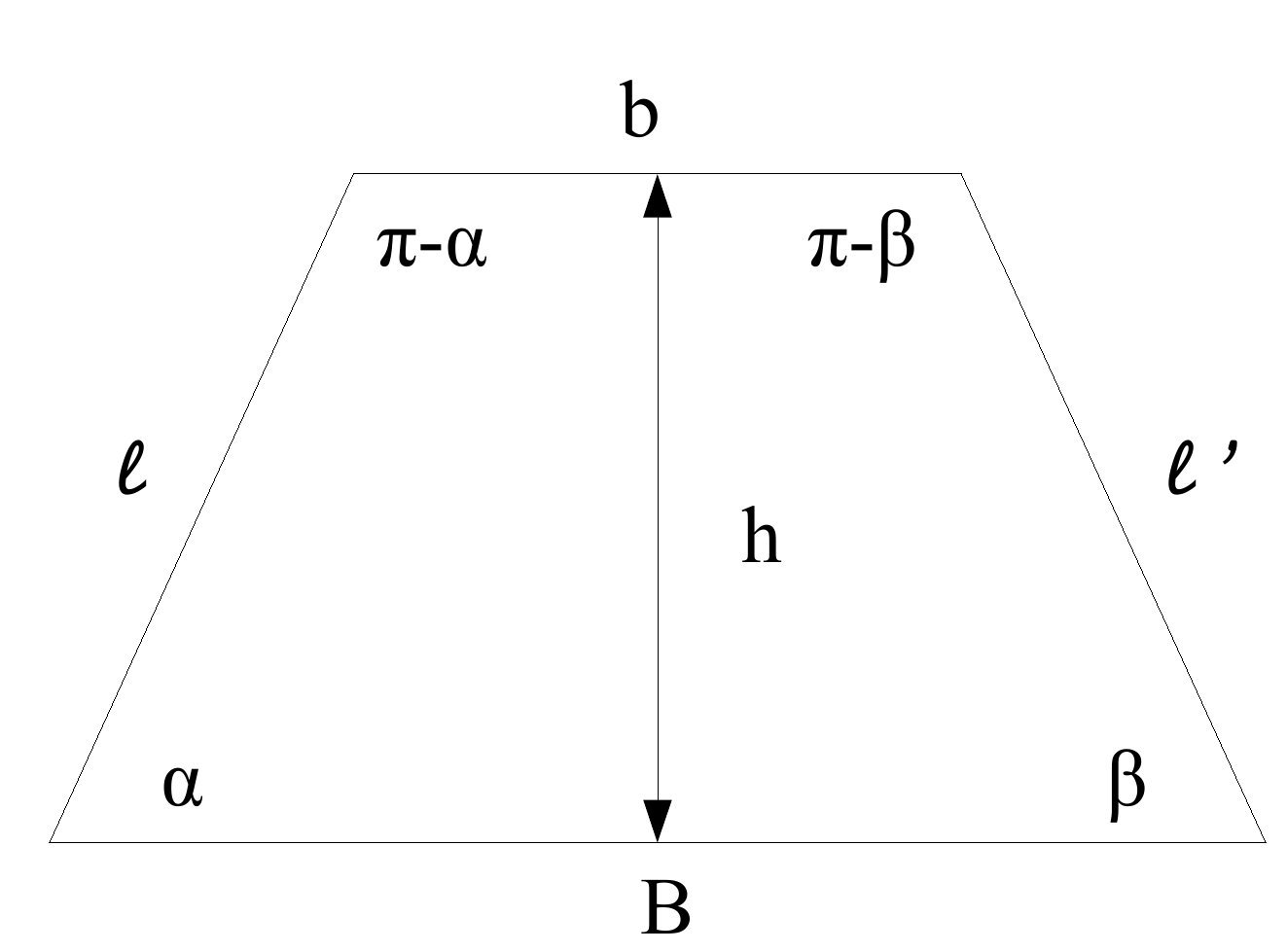} \end{figure} 

The heat trace ~\eqref{heat} together with the length of the shortest closed geodesic are determined by the spectrum.  We shall prove Theorem ~\ref{t1} by showing that the first three coefficients in the heat trace ~\eqref{heat} together with the length of the shortest closed geodesic are sufficient to uniquely determine an acute trapezoid.  What is the length of the shortest closed geodesic?    

\begin{lemma} 
 The length of the shortest closed geodesic of an acute trapezoid is twice the height, that is twice the distance from the base to the opposite parallel side.  
\end{lemma}

\begin{proof} 
There are two possibilities for the length of the shortest closed geodesic.  It is either $2h$ which corresponds to closed geodesics joining the parallel sides, or it is the perimeter of a triangle contained in the trapezoid.  Using the laws of  optics, if the shortest closed geodesic is a triangle, then its vertices lie on the legs and the base, and there is a segment joining each vertex on a leg to the opposite vertex at the base such that the segment meets the leg in a right angle.  Well, it turns out that this is impossible.  Try drawing a triangle with one side equal to such a segment from a vertex at the base to the leg meeting in a right angle, for example in Figure ~\ref{trap} from the vertex with angle $\beta$ to the opposite leg.  The sum of the angles of a triangle is $\pi$, so since one of the angles is $\frac{\pi}{2}$, the other two angles must sum to $\frac{\pi}{2}$.  However, the angle in the triangle at the vertex on the base measures at most $\beta$.  Since $\alpha + \beta < \frac{\pi}{2}$, such a triangle is impossible!  So, the shortest closed geodesic cannot be a triangle contained in the trapezoid. 
\end{proof} 

Since the spectrum determines the wave trace and hence its singularities, the first positive singularity which is the length of the shortest closed geodesic, $2h$, is a spectral invariant.  What spectral invariants can we extract from the short time asymptotic expansion of the heat trace ~\eqref{heat}?  Well, the first two coefficients are the area $A$ and the perimeter $P$.  The area of a trapezoid with base $B$ and opposite parallel side of length $b$ is 
$$A = \frac{B+b}{2} h \implies B+b = \frac{2A}{h}.$$
The perimeter 
$$P = B+b + \ell + \ell',$$
where $\ell$, $\ell'$ are the lengths of the legs.  This shows that the spectrum uniquely determines $h$, $B+b$, and $\ell + \ell'$.  What about the angles?  Let $\alpha$ and $\beta$ denote the interior angles at the base $B$, assuming without loss of generality that $\alpha \geq \beta$.  Then 
$$\ell = h \csc \alpha, \quad \ell' = h \csc \beta, \quad \ell + \ell' =  h(\csc \alpha + \csc \beta),$$
so 
\begin{equation} \label{ell} \csc \alpha + \csc \beta = \frac{\ell + \ell'}{h}. \end{equation} 
Since the angles are $\alpha$, $\pi - \alpha$, $\beta$, and $\pi - \beta$, the constant term in  ~\eqref{heat} is 
\begin{equation} \label{aotrap} a_0 = \frac{\pi^2}{24} \left( \frac{1}{\alpha(\pi - \alpha)} + \frac{1}{\beta(\pi - \beta)} \right) - \frac{1}{12}. \end{equation} 
In the following lemma, we will prove that ~\eqref{ell} and ~\eqref{aotrap} uniquely determine the angles $\alpha$ and $\beta$.  This means that the spectrum uniquely determines the angles, the height, the area, and the perimeter which all together uniquely determine the acute trapezoid, up to rigid motions.  The proof of this lemma therefore completes the proof of Theorem ~\ref{t1} and also shows that the method of using both the heat and wave traces can become highly technical.  

\begin{lemma} Let $p,q$ be real numbers. Then the solution of the system of equations
\begin{equation} \label{array} 
\left\{
\begin{array}{l}
\csc(\alpha)+\csc(\beta)=p\\
(\alpha(\pi-\alpha))^{-1}+(\beta(\pi-\beta))^{-1}=q
\end{array}
\right.
\end{equation} 
if it exists, must be unique  for $0< \beta \leq \alpha \leq\pi/2$.
\end{lemma}

\begin{proof}
First, let's use the second equation to show that each $\alpha$ uniquely determine a $\beta$.  Solving the second equation for $\beta$ in terms of $\alpha$ and $q$ leads to a quadratic equation in $\beta$, 
$$\beta^2 (1-q\alpha(\pi - \alpha)) + \beta (\pi(q\alpha (\pi - \alpha) -1)) - \alpha(\pi - \alpha) =0.$$
By the quadratic formula the solutions are 
$$\beta = \frac{\pi}{2} \pm \sqrt{ \frac{\pi^2}{4} + \frac{\alpha(\pi - \alpha)}{1-q\alpha(\pi - \alpha)}}.$$
Since the trapezoid is acute, $\beta \leq \frac{\pi}{2}$, so indeed $\alpha$ and $q$ uniquely determine  
\begin{equation} \label{eq-beta} 
\beta=\beta(\alpha) = \frac{\pi}{2} - \sqrt{ \frac{\pi^2}{4} + \frac{\alpha(\pi - \alpha)}{1-q\alpha(\pi - \alpha)}}.
\end{equation} 
We can prove the lemma if we prove that the function 
\begin{equation} \label{def-g} 
g(\alpha) := \csc(\alpha)+\csc(\beta(\alpha)) \end{equation} 
has unique solution $\alpha$ for any given $p$.  Analyzing this function directly is problematic, due to the presence of the unknown constant $q$ in the expression for $\beta$ ~\eqref{eq-beta}.  So, we'd like to relate the function $g$ to an explicit function of one variable.  A good way to get rid of unwanted, unknown constants is to differentiate.  Implicit differentiation in the second equation of ~\eqref{array} gives
$$\beta'(\alpha) = - \frac{((\alpha(\pi-\alpha))^{-1} )' }{((\beta(\pi-\beta))^{-1} )'},$$
and so the derivative
\begin{equation}\label{2}
g'(\alpha) = -\csc(\alpha)\cot(\alpha)+\csc(\beta(\alpha))\cot(\beta(\alpha))\cdot\frac{((\alpha(\pi-\alpha))^{-1})'}
{((\beta(\pi-\beta))^{-1})'}.
\end{equation}

Let's see if we can relate $g'(\alpha)$ to the following function  
\[
f(\alpha):=\frac{\csc(\alpha) \cot(\alpha)}{((\alpha(\pi-\alpha))^{-1})'} = \frac{\alpha^2(\pi-\alpha)^2\cos\alpha}{(2\alpha-\pi)\sin^2\alpha}.
\]

Since 
$$\frac{g'(\alpha)}{((\alpha(\pi-\alpha))^{-1})'} = - \frac{\csc(\alpha) \cot(\alpha)}{((\alpha(\pi-\alpha))^{-1})'} + \frac{\csc(\beta(\alpha)) \cot(\beta(\alpha))}{((\beta(\pi-\beta))^{-1})'},$$
we see that 
\begin{equation} \label{g'} g'(\alpha) = \frac{\pi - 2\alpha}{\alpha^2 (\pi - \alpha)^2} \left(f(\alpha)  -f(\beta) \right) \end{equation} 

Although 
$$f'(\alpha) = - \frac{2 (\pi - \alpha)^2 \alpha^2 \cot(\alpha) \csc(\alpha)}{(2\alpha - \pi)^2} + \frac{2 (\pi - \alpha)^2 \alpha \cot(\alpha) \csc(\alpha)}{2\alpha - \pi} $$
$$- \frac{2(\pi - \alpha) \alpha^2 \cot(\alpha) \csc(\alpha)}{2\alpha - \pi} - \frac{(\pi - \alpha)^2 \alpha^2 \cot^2 (\alpha) \csc(\alpha)}{2\alpha - \pi} - \frac{(\pi-\alpha)^2 \alpha^2 \csc^3 (\alpha)}{2\alpha - \pi},$$
it turns out that the logarithmic derivative of $f$ is pleasantly simple 
$$\frac{f'(\alpha)}{f(\alpha)} = \frac{2}{\pi - 2\alpha} + \frac{2}{\alpha} + \frac{2}{\alpha-\pi}  - 2 \cot (\alpha) - \tan (\alpha).$$ 
If we are able to prove that 
\begin{equation} \label{logf'} \log(f(\alpha))' < 0 \quad \forall \alpha \in \left( 0, \frac{\pi}{2} \right), \end{equation} 
it follows that $f$ is a strictly monotone function on $(0, \pi/2)$.  Since we assumed $\alpha \leq \beta < \frac{\pi}{2}$ by ~\eqref{g'} we see that if $f$ is strictly monotone, then $g'(\alpha) \neq 0$ for all $\alpha \in (0, \beta)$.  We began by proving that each $\alpha$ uniquely determines a $\beta$ which shows that ~\eqref{def-g} has a unique solution $\alpha$ for any given $p$, and consequently if a solution exists to ~\eqref{array}, then it is unique.  So, the lemma is reduced to proving the following.  

\textbf{Claim:  }  The function 
\begin{equation} \label{u} u(\alpha) := \frac{f'(\alpha)}{f(\alpha)} < 0 \quad \forall \alpha \in \left( 0, \frac{\pi}{2} \right). \end{equation}

A graph produced by Mathematica numerically proves the claim (see Figure ~\ref{graphf}).  However,  some readers may be interested to see that it is possible to prove this ``by hand'' using carefully chosen algebraic manipulations and a bit of calculus.   
\begin{figure} \caption{Graph of $(\log f)'$}\label{graphf} \includegraphics[width=300pt]{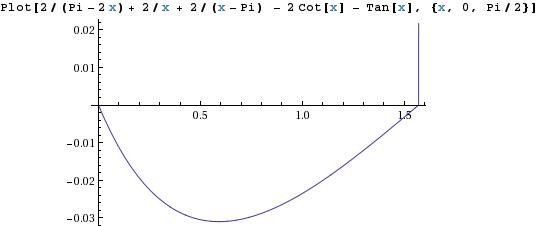} \end{figure} 


We compute 
$$u''(\alpha) = \frac{16}{(\pi - 2\alpha)^3} + \frac{4}{\alpha^3} + \frac{4}{(\alpha- \pi)^3} - \frac{4 \cos(\alpha)}{\sin^3 (\alpha)} - \frac{2 \sin(\alpha)}{\cos^3 (\alpha)}.$$
Using the trigonometric identities 
$$\sin\left( \frac{\pi}{2} - \alpha \right) = \cos (\alpha), \quad \cos \left( \frac{\pi}{2} - \alpha \right) = \sin(\alpha),$$
we see that
$$u''(\alpha) = \frac{4}{(\alpha - \pi)^3} + 4 v(\alpha) + 2 v\left( \frac{\pi}{2} - \alpha \right),$$
where for ease of notation
\begin{equation} \label{v} v(\alpha) := \frac{1}{\alpha^3} - \frac{\cos(\alpha)}{\sin^3 (\alpha)}. \end{equation}

We next compute 
$$v'(\alpha) = - \frac{3}{\alpha^4} + \frac{1}{\sin^2(\alpha)} + \frac{3 \cos^2 (\alpha)}{\sin^4 (\alpha)} = - \frac{3}{\alpha^4} + \frac{1}{\sin^2(\alpha)} + \frac{ 3(1-\sin^2(\alpha))}{\sin^4(\alpha)} $$
$$=  - \frac{3}{\alpha^4} - \frac{2}{\sin^2 (\alpha)} + \frac{3}{\sin^4 (\alpha)}= \frac{ - 3 \sin^4 (\alpha) + 3 \alpha^4 - 2 \alpha^4 \sin^2(\alpha)}{\alpha^4 \sin^4 (\alpha) }.$$
It follows that 
$$v'(\alpha) \geq 0 \iff \left( \frac{\sin(\alpha)}{\alpha} \right)^4 \leq 1 - \frac{2 \sin^2 (\alpha)}{3}.$$
By the power series expansion for sine, 
$$\sin(\alpha) \leq \alpha - \frac{\alpha^3}{6} + \frac{\alpha^5}{120},$$
so 
$$\left( \frac{\sin(\alpha)}{\alpha} \right)^4 \leq \left( 1 - \frac{\alpha^2}{6} + \frac{\alpha^4}{120} \right)^4,$$
and 
$$1 - \frac{2 \sin^2 (\alpha)}{3} \geq 1 - \frac{2}{3} \alpha^2 \left( 1 - \frac{\alpha^2}{6} + \frac{\alpha^4}{120} \right)^2.$$
We can prove that $v'(\alpha) \geq 0$ if we can show that 
$$ \left( 1 - \frac{\alpha^2}{6} + \frac{\alpha^4}{120} \right)^4 \leq  1 - \frac{2}{3} \alpha^2 \left( 1 - \frac{\alpha^2}{6} + \frac{\alpha^4}{120} \right)^2,$$
which is equivalent to 
$$\left( 1 - \frac{\alpha^2}{6} + \frac{\alpha^4}{120} \right)^4 +  \frac{2}{3} \alpha^2 \left( 1 - \frac{\alpha^2}{6} + \frac{\alpha^4}{120} \right)^2 - 1 \leq 0.$$
Well, this is none other than a quadratic equation in 
$$t^2 := \left(1 - \frac{\alpha^2}{6} + \frac{\alpha^4}{120}\right)^2.$$  
The quadratic formula shows that 
$$t^4 + \frac{2}{3}\alpha^2 t^2 - 1 = 0 \iff t^2 = - \frac{ \alpha^2}{3} \pm \sqrt{ \frac{ \alpha^2}{9} + 1}.$$
Pausing for a moment to think about the graph of the function $t^4 + \frac{2}{3}\alpha^2 t^2 - 1 $, we see that it is non-positive if and only if 
$$- \frac{\alpha^2}{3} - \sqrt{ \frac{\alpha^4}{9} + 1} \leq t^2 \leq - \frac{\alpha^2}{3} + \sqrt{ \frac{\alpha^4}{9} + 1}.$$
Now,  letting $(1 - \alpha^2/6 + \alpha^4/120)$ play the role of $t$, 
$$\left( 1 - \frac{\alpha^2}{6} + \frac{\alpha^4}{120} \right)^4 \leq  1 - \frac{2}{3} \alpha^2 \left( 1 - \frac{\alpha^2}{6} + \frac{\alpha^4}{120} \right)^2$$
$$ \iff - \frac{\alpha^2}{3} - \sqrt{ \frac{\alpha^4}{9} + 1} \leq \left( 1 - \frac{\alpha^2}{6} + \frac{\alpha^4}{120} \right)^2 \leq - \frac{\alpha^2}{3} + \sqrt{ \frac{\alpha^4}{9} + 1}.$$ 
Clearly the left inequality always holds.  So it suffices to prove the right inequality.  Since 
$$\left( 1 - \frac{\alpha^2}{6} + \frac{\alpha^4}{120} \right)^2 \leq 1 - \frac{\alpha^2}{3} + \frac{2 \alpha^4}{45},$$
and 
$$\left( 1 + \frac{2 \alpha^4}{45} \right)^2 \leq 1 + \frac{\alpha^4}{9} \implies 1 + \frac{2\alpha^4}{45} \leq \sqrt{ \frac{\alpha^4}{9} +1},$$
we have 
$$\left( 1 - \frac{\alpha^2}{6} + \frac{\alpha^4}{120} \right)^2 \leq - \frac{\alpha^2}{3} +1+ \frac{2 \alpha^4}{45} \leq -\frac{\alpha^2}{3} + \sqrt{ \frac{\alpha^4}{9} + 1}.$$
This shows that $v'(\alpha) \geq 0$.  Since 
$$u''(\alpha) = \frac{4}{(\alpha- \pi)^3} + 4 v(\alpha) + 2 v\left( \frac{\pi}{2} - \alpha \right),$$
for 
$$\alpha \leq \frac{\pi}{4} \implies u''(\alpha) \geq \frac{4}{(\pi/4 - \pi)^3} + 2 v\left( \frac{\pi}{4} \right) > 0,$$
and for 
$$\alpha > \frac{\pi}{4} \implies u''(\alpha) \geq - \frac{32}{\pi^3} + 4 v\left( \frac{\pi}{4} \right) > 0.$$
Consequently $u$ is convex.  Since $u(0) = u(\pi/2) = 0$, $u(\alpha) < 0$ on $(0, \pi/2)$.  
\end{proof}  

\begin{ex} Find an example of two different acute trapezoids with identical area, perimeter, and constant term $a_0$ in the heat trace ~\eqref{heat}.  This shows that it is impossible to use the first three terms in the heat trace expansion to prove that isospectral trapezoids are congruent.  
       \end{ex}

\section{Hearing the regular $n$-gon} \label{s4}
\subsection{Proof of Theorem~\ref{t2}}
We begin with a short proof of Theorem~\ref{t2} based on Proposition~\ref{p1}.  
\begin{proof}[Proof of Theorem~\ref{t2}]
Let's assume that for a fixed $n \geq 3$ there is an $n$-gon $Q$ such that 
$$\lambda_k (Q) = \lambda_k (R_n), \quad \forall k \geq 1,$$
where $R_n$ is a regular $n$-gon.  In \S~\ref{s2} we proved that the function 
$$f(Q) = \frac{ \textrm{ Area of } Q}{ (\textrm{Perimeter of }Q)^2}$$
is maximized among all $n$-gons by a regular $n$-gon.  Since the spectrum determines the heat trace, by ~\eqref{heat} it follows that $Q$ and $R_n$ have the same area and perimeter so, 
$$f(Q) = f(R_n).$$
Consequently, $Q$ is also regular, and since $Q$ and $R_n$ have the same perimeter and area, they are the same regular $n$-gon up to a rigid motion.  
\end{proof} 
 
\subsection{Proof of the weak P\'olya-Szeg{\H{o}} conjecture}
The proof of Theorem ~\ref{t3} is based on Proposition ~\ref{pr-a}, Theorem ~\ref{t2}, and Proposition ~\ref{propcd}.  

\begin{proof}[Proof of Theorem~\ref{t3}] 
For the sake of contradiction, let's assume that there exist convex $n$-gons $\{\Omega_k\}_{k=1} ^\infty$ with 
\begin{equation} \label{contra} \lambda_i (\Omega_k) = \lambda_i (P_k), \quad \forall i < k,  \end{equation} 
where $P_k$ is a regular $n$-gon, and each $\Omega_k$ is \em not \em regular.  The scaling property for the eigenvalues  shows that ~\eqref{contra} is equivalent to 
\begin{equation} \label{contra2} \lambda_i (\tilde{\Omega}_k) = \lambda_i (P), \quad \forall i < k, \end{equation} 
where $P$ is the regular $n$-gon with diameter equal to 1, and $\tilde{\Omega}_k$ is $\Omega_k$ scaled by $d_k^{-1}$, where $d_k$ is the diameter of $P_k$.  Slightly abusing notation, we'll write $\Omega_k$ rather than $\tilde{\Omega}_k$.  

By Proposition ~\ref{pr-a} since for $k > 3$ the domains $\Omega_k$ have the same first two eigenvalues, it follows that $\Omega_k$ can neither collapse nor explode as $k \to \infty$.  Passing to a sub-sequence if necessary, we assume that 
$$\Omega_k \to \Omega_0,$$
where $\Omega_0$ is a convex $j$-gon for some $3 \leq j \leq n$.  We know that $j \leq n$ because as the $n$-gons $\Omega_k \to \Omega_0$, they cannot magically obtain more sides, however some of the interior angles could tend toward $\pi$ causing the number of sides to decrease.  So, $j > n$ is impossible, but what if $j< n$?    The eigenvalues are continuous, and so the eigenvalues of $\Omega_k$ tend to those of  $\Omega_0$, and 
$$\lambda_i (P) = \lambda_i (\Omega_k) \to \lambda_i (\Omega_0), \quad \forall i < k, \quad \textrm{ as $k \to \infty$.}$$
Therefore, 
\begin{equation}\label{o0} \lambda_i (\Omega_0) = \lambda_i (P), \quad \forall i. \end{equation} 
The area and perimeter of $\Omega_0$ are determined by the spectrum which coincides with that of $P$, so $f(\Omega_0) = f(P)$.  If $\Omega_0$ is a $j$-gon for some $j < n$, then 
$$f(P) = f(\Omega_0) \leq f(P_j) < f(P),$$ 
where $P_j$ is a regular $j$-gon.  That's rubbish!  So we see that $\Omega_0$ is also an $n$-gon.  By Theorem~\ref{t1} and ~\eqref{o0}, $\Omega_0 \cong P$.  

By Proposition~\ref{propcd}, we can parametrize  $\{\Omega_k \}_{k=0} ^\infty$ by $\{\epsilon_k\}_{k=0} ^\infty \subset \R^{2n-3}$.  Since $\Omega_k \to \Omega_0$ as $k \to \infty$, 
$$\lim_{k \to \infty} \epsilon_k = \epsilon_0,$$ 
and the set $\{\epsilon_k\}_{k=0} ^\infty$ is a compact subset of $\R^{2n-3}$.  By Theorem~\ref{th-cd}, there exists $N = N(n)$ such that if the first $N$ eigenvalues of $\D$ with respect to the metric $g(\epsilon_k)$ coincide with the first $N$ eigenvalues of $g(\epsilon_0)$, then all the eigenvalues coincide.  Since the eigenvalues of $\D$ with respect to the metric $g(\epsilon_k)$ are identical to the eigenvalues of $\Omega_k$, this shows that for all $k > N$, 
$$\lambda_i (\Omega_k) = \lambda_i (\Omega_0) = \lambda_i (P), \quad \forall i \leq N \implies \lambda_i (\Omega_k) = \lambda_i (P), \quad \forall i.$$
By Theorem~\ref{t1}, this implies that 
$$\Omega_k \cong P \quad \forall k > N,$$
which contradicts our assumption that $\Omega_k$ is \em not \em regular for all $k \in \N$.  

We conclude that there exists an $N$ which depends on $n$ such that if the first $N$ eigenvalues of a convex $n$-gon $\Omega$ coincide with those of a regular $n$-gon, then $\Omega$ is congruent to that regular $n$-gon.  
\end{proof} 

\section{Conjectures and open problems} \label{s5}
The P\'olya-Szeg{\H{o}} Conjecture would indicate that $N = N(n)$ in Theorem~\ref{t3} may be taken equal to 1, however, that conjecture assumes the polygons all have fixed area.  Since Theorem~\ref{t3} holds without the area-normalization assumption, we propose that the original conjecture may be strengthened as follows.  

\begin{conj}[Strong P\'olya-Szeg\H{o} Conjecture] The number $N= N(n)$ in Theorem~\ref{t3} may be taken equal to $1$.  
\end{conj} 

For the special case of triangles, Antunes and Freitas made the natural conjecture in \cite{af}, that the first three eigenvalues uniquely determine a triangle.  They provided a vast amount of supporting numerical data, so that one may consider the conjecture to be ``numerically'' a theorem.  Since three parameters uniquely determine a triangle, this conjecture may seem rather obvious at first glance, and one may ask more generally whether any three eigenvalues would suffice.  Intriguingly, [ibid] showed that the first, second, and fourth eigenvalues \em cannot \em uniquely determine a triangle, so the conjecture appears to be more subtle than one might expect.  

\begin{conj}[Antunes-Freitas] If the first three eigenvalues of two triangles coincide, then they are identical up to a rigid motion. 
\end{conj} 

One can also consider isospectral problems for the \em length spectrum, \em the set of lengths of closed geodesics.  It's usually considered a good idea to make pure mathematical conjectures based on observations in physics or nature.  In this setting, we are reminded of bats who use echolation to determine their location from objects and prey.  A bat emits a sound (which is generally inaudible to the human ear) and remarks the time(s) at which the sound is reflected back.  It is only possible for the bat to detect a finite amount of return times, which mathematically correspond to the lengths of finitely many closed geodesics.  Inspired by nature we make the following ``bat conjecture.''  


\begin{conj} \label{bats} For each $n \geq 3$ there exists $N = N(n)$ such that if the lengths of $N$ primitive closed geodesics of two convex $n$-gons coincide, then they are identical up to a rigid motion.  
\end{conj} 

Finally two natural questions arise from our work, the more tractable of which is the following.  

\begin{question}
Is a trapezoid uniquely determined by its spectrum?  Do there exist isospectral trapezoids which are not isometric?  
\end{question}

More generally, we are very curious to know the answer to the following.  

\begin{question}
Can one hear the shape of a convex $4$-sided drum? 
\end{question}

\section*{Acknowledgements} 
The first author is supported by NSF grant DMS-12-06748, and the second author gratefully acknowledges the support of the Max Planck Institut f\"ur Mathematik in Bonn.


\begin{bibdiv}

\begin{biblist}
\bib{ac}{article}{ author = {Andrews, Ben}, author={Clutterbuck, Julie}, title={Proof of the fundamental gap conjecture}, journal= {J. Amer. Math. Soc.}, volume={24}, year={2011}, pages={899--916}}  


\bib{afn}{article}{
   author={Antunes, Pedro},
   author={Freitas, Pedro},
   title={New bounds for the principal Dirichlet eigenvalue of planar
   regions},
   journal={Experiment. Math.},
   volume={15},
   date={2006},
   number={3},
   pages={333--342},
   issn={1058-6458},
   review={\MR{2264470 (2007e:35039)}},
}

\bib{af}{article}{
   author={Antunes, Pedro R. S.},
   author={Freitas, Pedro},
   title={On the inverse spectral problem for Euclidean triangles},
   journal={Proc. R. Soc. Lond. Ser. A Math. Phys. Eng. Sci.},
   volume={467},
   date={2011},
   number={2130},
   pages={1546--1562},
   issn={1364-5021},
   review={\MR{2795790 (2012e:65242)}},
   doi={10.1098/rspa.2010.0540},
}

\bib{vdb}{article}{
   author={van den Berg, M.},
   author={Srisatkunarajah, S.},
   title={Heat equation for a region in ${\bf R}^2$ with a polygonal
   boundary},
   journal={J. London Math. Soc. (2)},
   volume={37},
   date={1988},
   number={1},
   pages={119--127},
   issn={0024-6107},
   review={\MR{921750 (89e:35062)}},
   doi={10.1112/jlms/s2-37.121.119},
}

\bib{iso}{article}{
   author={Bl{\aa}sj{\"o}, Viktor},
   title={The isoperimetric problem},
   journal={Amer. Math. Monthly},
   volume={112},
   date={2005},
   number={6},
   pages={526--566},
   issn={0002-9890},
   review={\MR{2142606 (2006a:51017)}},
   doi={10.2307/30037526},
}


\bib{buser}{article}{author={Buser, Peter}, title={Isospectral Riemann surfaces}, journal={Ann. Inst. Fourier}, volume={36}, year={1986}, pages={167--192}} 

\bib{c}{unpublished}{
   author={Chang, Pei-Kun},
   title={Thesis},
   note={Thesis (Ph.D.)--University of Pennsylvania},
   date={1988},
}


\bib{cd}{article}{
   author={Chang, Pei-Kun},
   author={DeTurck, Dennis},
   title={On hearing the shape of a triangle},
   journal={Proc. Amer. Math. Soc.},
   volume={105},
   date={1989},
   number={4},
   pages={1033--1038},
   issn={0002-9939},
   review={\MR{953738 (89h:58194)}},
   doi={10.2307/2047071},
}
	
	
	\bib{chapman}{article}{author={Chapman, S. J.}, title={Drums that sound the same}, journal={Amer. Math. Monthly}, volume={102}, year={1995}, pages={124--138}} 

\bib{chavel}{book}{
   author={Chavel, Isaac},
   title={Eigenvalues in Riemannian geometry},
   series={Pure and Applied Mathematics},
   volume={115},
   note={Including a chapter by Burton Randol;
   With an appendix by Jozef Dodziuk},
   publisher={Academic Press Inc.},
   place={Orlando, FL},
   date={1984},
   pages={xiv+362},
   isbn={0-12-170640-0},
   review={\MR{768584 (86g:58140)}},
}

\bib{cour-hil}{book}{
   author={Courant, R.},
   author={Hilbert, D.},
   title={Methoden der Mathematischen Physik. Vols. I, II},
   publisher={Interscience Publishers, Inc., N.Y.},
   date={1943},
   pages={xiv+469 pp., xiv+549},
   review={\MR{0009069 (5,97b)}},
}

\bib{sc}{book}{
   author={Driscoll, Tobin A.},
   author={Trefethen, Lloyd N.},
   title={Schwarz-Christoffel mapping},
   series={Cambridge Monographs on Applied and Computational Mathematics},
   volume={8},
   publisher={Cambridge University Press},
   place={Cambridge},
   date={2002},
   pages={xvi+132},
   isbn={0-521-80726-3},
   review={\MR{1908657 (2003e:30012)}},
   doi={10.1017/CBO9780511546808},
}

		



\bib{dg}{article}{
   author={Duistermaat, J. J.},
   author={Guillemin, V. W.},
   title={The spectrum of positive elliptic operators and periodic
   bicharacteristics},
   journal={Invent. Math.},
   volume={29},
   date={1975},
   number={1},
   pages={39--79},
   issn={0020-9910},
   review={\MR{0405514 (53 \#9307)}},
}


\bib{dur}{book}{
   author={Durso, Catherine},
   title={On the inverse spectral problem for polygonal domains},
   note={Thesis (Ph.D.)--Massachusetts Institute of Technology},
   publisher={ProQuest LLC, Ann Arbor, MI},
   date={1988},
   pages={(no paging)},
   review={\MR{2941198}},
}


\bib{fed}{article}{
   author={Fedosov, B. V.},
   title={Asymptotic formulae for the eigenvalues of the Laplace operator in
   the case of a polygonal domain},
   language={Russian},
   journal={Dokl. Akad. Nauk SSSR},
   volume={151},
   date={1963},
   pages={786--789},
   issn={0002-3264},
   review={\MR{0157095 (28 \#335)}},
}

\bib{frtri}{article}{author= {Freitas, Pedro}, title={Precise bounds and asymptotics for the first Dirichlet eigenvalue of triangles and rhombi}, journal={Jour. Funct. Anal.}, volume={251}, year={2007}, pages={376--398}} 

\bib{fried}{article}{author={Friedlander, F. G.}, title={On the wave equation in plane regions with polygonal boundary}, journal={Advances in microlocal analysis}, place={Lucca, 1985}, year={1986}, pages={135--150}} 

\bib{strip}{article}{author={Friedlander, L}, author={Solomyak, M.}, title={On the Spectrum of the Dirichlet Laplacian in a Narrow Strip}, journal={Israel J. Math.}, volume={170}, year={2009}, pages={337--354}}

\bib{gww}{article}{
   author={Gordon, Carolyn},
   author={Webb, David L.},
   author={Wolpert, Scott},
   title={One cannot hear the shape of a drum},
   journal={Bull. Amer. Math. Soc. (N.S.)},
   volume={27},
   date={1992},
   number={1},
   pages={134--138},
   issn={0273-0979},
   review={\MR{1136137 (92j:58111)}},
   doi={10.1090/S0273-0979-1992-00289-6},
}

\bib{gww1}{article}{
   author={Gordon, C.},
   author={Webb, D.},
   author={Wolpert, S.},
   title={Isospectral plane domains and surfaces via Riemannian orbifolds},
   journal={Invent. Math.},
   volume={110},
   date={1992},
   number={1},
   pages={1--22},
   issn={0020-9910},
   review={\MR{1181812 (93h:58172)}},
   doi={10.1007/BF01231320},
}

\bib{gm}{unpublished}{
author={Grieser, D.},
author={Maronna, S.},
title={One can hear the shape of a triangle},
note={arXiv 1208.3163},
year={2012},}



\bib{kac}{article}{
   author={Kac, Mark},
   title={Can one hear the shape of a drum?},
   journal={Amer. Math. Monthly},
   volume={73},
   date={1966},
   number={4},
   pages={1--23},
   issn={0002-9890},
   review={\MR{0201237 (34 \#1121)}},
}

\bib{crm}{unpublished}{author={Lu, Zhiqin}, author={Rowlett, Julie}, title={The fundamental gap and one-dimensional collapse}, note={to appear in Proc. of the C.~R.~M., \url{http://arxiv.org/abs/0810.4937}}, year={2013}} 


\bib{ms}{article}{
   author={McKean, H. P., Jr.},
   author={Singer, I. M.},
   title={Curvature and the eigenvalues of the Laplacian},
   journal={J. Differential Geometry},
   volume={1},
   date={1967},
   number={1},
   pages={43--69},
   issn={0022-040X},
   review={\MR{0217739 (36 \#828)}},
}

\bib{milnor}{article}{ author = {Milnor, John}, title={Eigenvalues of the Laplace operator on certain manifolds}, journal={Proc. Nat. Adad. Sci. USA}, volume={51}, number={4}, year={1964}, pages={542}}

\bib{pw}{article}{
   author={Payne, L. E.},
   author={Weinberger, H. F.},
   title={Some isoperimetric inequalities for membrane frequencies and
   torsional rigidity},
   journal={J. Math. Anal. Appl.},
   volume={2},
   date={1961},
   pages={210--216},
   issn={0022-247x},
   review={\MR{0149735 (26 \#7220)}},
}

\bib{perron}{article}{ author={Perron, Oskar}, title={Zur Existenzfrage eines Maximums oder Minimums}, journal={Jahresber. Deutsch. Math.-Verein.}, volume={22}, year={1913}, pages={140--144}} 

\bib{pl}{article}{
   author={Pleijel, {\AA}ke},
   title={A study of certain Green's functions with applications in the
   theory of vibrating membranes},
   journal={Ark. Mat.},
   volume={2},
   date={1954},
   pages={553--569},
   issn={0004-2080},
   review={\MR{0061257 (15,798g)}},
}

\bib{po}{book}{
   author={P{\'o}lya, G.},
   author={Szeg{\"o}, G.},
   title={Isoperimetric Inequalities in Mathematical Physics},
   series={Annals of Mathematics Studies, no. 27},
   publisher={Princeton University Press},
   place={Princeton, N. J.},
   date={1951},
   pages={xvi+279},
   review={\MR{0043486 (13,270d)}},
}


\bib{rose}{book}{
   author={Rosenberg, Steven},
   title={The Laplacian on a Riemannian manifold},
   series={London Mathematical Society Student Texts},
   volume={31},
   note={An introduction to analysis on manifolds},
   publisher={Cambridge University Press},
   place={Cambridge},
   date={1997},
   pages={x+172},
   isbn={0-521-46300-9},
   isbn={0-521-46831-0},
   review={\MR{1462892 (98k:58206)}},
   doi={10.1017/CBO9780511623783},
}


\bib{stein}{article}{ author={Steiner, Jakob}, title={Einfache Beweise der isoperimetrischen Haupts\"atze}, journal={J. Reine Angew. Math.}, volume={18}, year={1838}, pages={281--296}}

\bib{sun}{article}{
   author={Sunada, Toshikazu},
   title={Riemannian coverings and isospectral manifolds},
   journal={Ann. of Math. (2)},
   volume={121},
   date={1985},
   number={1},
   pages={169--186},
   issn={0003-486X},
   review={\MR{782558 (86h:58141)}},
   doi={10.2307/1971195},
}
	
\bib{weyl}{article}{
   author={Weyl, Hermann},
   title={Das asymptotische Verteilungsgesetz der Eigenwerte linearer
   partieller Differentialgleichungen (mit einer Anwendung auf die Theorie
   der Hohlraumstrahlung)},
   language={German},
   journal={Math. Ann.},
   volume={71},
   date={1912},
   number={4},
   pages={441--479},
   issn={0025-5831},
   review={\MR{1511670}},
   doi={10.1007/BF01456804},
}

\bib{witt}{article}{ author= {Witt, E.}, title={Eine Identit\"at zwischen Modulformen zweiten Grades}, journal={Abh. Math. Sem. Univ. Hamburg}, volume={14}, year={1941}, pages={323--337}} 

\end{biblist}
\end{bibdiv}
\end{document}